\documentclass[11pt, a4paper]{amsart}

\calclayout

\usepackage[table]{xcolor}
\usepackage{tikz}
\usepackage{amsmath}
\usepackage{amssymb}
\usepackage{amsthm}
\usepackage{amsfonts}

\usepackage{graphicx}
\usepackage{subcaption}
\usepackage{cprotect}
\usepackage{cite}

\usepackage[T1]{fontenc}

\usepackage{enumerate}

\usepackage{todonotes, changes}
\usepackage{comment}

\usepackage[shortlabels]{enumitem}
\newlist{mycases}{enumerate}{1}
\setlist[mycases]{
    label=({\roman*}), % How it looks in the list: Case 1:
    ref=(\roman*),                  % How the number is stored for \ref
    leftmargin=2cm
}

\newcommand{\orbit}{\mathcal{Z}}

\newcommand{\RR}{\mathbb{R}}

\newcommand{\ZZ}{\mathbb{Z}}

\newcommand{\K}{\mathcal{K}}

\newcommand{\conv}{\mathrm{conv}}

\newcommand{\dist}{\mathrm{dist}}

\newcommand{\inter}{\mathrm{int}}
\newcommand{\relint}{\mathrm{relint}}
\newcommand{\cl}{\mathrm{cl}}
\DeclareMathOperator{\vol}{vol}

\newcommand{\ball}{\mathbb{B}}

\newcommand{\GL}{\mathrm{GL}}

\DeclareMathOperator{\expo}{expo}
\DeclareMathOperator{\fcone}{fcone}

\newcommand{\origin}{o}
\newcommand{\wozero}{\setminus\{\origin\}}

\usepackage{hyperref}
\hypersetup{
colorlinks=true,
  citecolor=black, linkcolor=black, urlcolor=black}

\title[Minimal covering bodies]{Minimal covering bodies and Brunn-Minkowski type inequalities for the covering radius}
\author[G.\ Codenotti]{Giulia Codenotti}
\author[A.\ Freyer]{Ansgar Freyer}
\author[K.\ Krivoku\'{c}a]{Katarina Krivoku\'{c}a}
\address{FU Berlin, AG Diskrete Geometrie und Topologische Kombinatorik, Arnimallee 2, 14195 Berlin}
\email{\{giulia.codenotti, a.freyer, katarina.krivokuca\}@fu-berlin.de} 
\date{}
\thanks{G.\ Codenotti and A.\ Freyer are funded by the Deutsche Forschungsgemeinschaft (DFG, German Research Foundation) -
539867386.\\
\indent K.\ Krivoku\' ca is funded by the DFG under Germany´s Excellence Strategy– The Berlin Mathematics Research Center MATH+ (EXC-2046/2, project ID: 390685689).
}

\numberwithin{equation}{section}

\usepackage{mathtools}
\mathtoolsset{showonlyrefs=true}

\theoremstyle{plain}
\newtheorem{theorem}{Theorem}[section]
\newtheorem{lemma}[theorem]{Lemma}
\newtheorem{corollary}[theorem]{Corollary}

\newtheorem{proposition}[theorem]{Proposition}
\newtheorem*{proposition*}{Proposition}

\newtheorem{remark}[theorem]{Remark}

\newtheorem*{claim*}{Claim}
\newtheorem*{remark*}{Remark}

\begin{document}
\maketitle

\begin{abstract}
Inclusion minimal convex bodies $K$ with the property that the integer translates of $K$ cover the space are studied. Such bodies are referred to as minimal covering bodies and it is shown that, while they are not necessarily tiles, they are polytopes with at least $2d$ facets, if $d$ is the dimension of $K$. Moreover, minimal covering bodies are related to covering properties of Minkowski combinations of convex bodies. Two sharp Brunn Minkowski type inequalities are established for the covering radius of the Minkowski sum of planar convex bodies. 
\end{abstract}

\section{Introduction}

%Motivation: Extremal structures, tilings
%Main results: Polytope with 2d many facets, Polytopal fundamental domains, Brunn-Minkowski type results.

While it is commonly known that any planar covering body contains a tiling convex hexagon (or quadrilateral), a construction of Xue and Zong \cite{xuezong} shows that in dimensions $d\geq 3$ there exist covering convex bodies $K$ that are not tiles but do not admit a convex body $K'\subsetneq K$ which would also be covering. We refer to convex bodies with the latter property as \emph{minimal covering.} 
It has been communicated to the authors that Yanlu Lian and Fei Xue investigate minimal covering bodies independently from this research.

In a similar vein, there has been a substantial amount of research on Voronoi-type cells and metric bisectors induced by convex bodies \cite{blomerkohn, criadojoswigsantos, jaljochemko, xuezong}. A common thread in these investigations is the crucial role of the boundary structure of the convex body at hand. Since the boundaries of general convex bodies can behave rather pathologically, many of the results are stated for polytopes or sufficiently smooth or strictly convex bodies. Fortunately, the minimal covering bodies turn out to be polytopes.

\begin{theorem}[Theorems \ref{thm:polytopes} and \ref{thm:lower_bound}]
\label{thm:main}
Let $K\subset\RR^d$ be a minimal covering body. Then, $K$ is a polytope with at least $2d$ facets.
\end{theorem}

Yanlu Lian and Fei Xue also obtained a proof for polytopality, as well as various non-trivial constructions for minimal covering bodies, to appear in an upcoming paper.

Regarding the combinatorial structure of minimal covering polytopes, the situation is less in control than in the case of tiles, where we have a characterization due to McMullen \cite{mcmullen} and a complete enumeration of all combinatorial types up to dimension 5 \cite{vor5,garber}. For minimal covering bodies, there are at least as many combinatorial types as there are combinatorial types of regular subdivisions of the permutahedron in dimension $d-1$, see Section \ref{sec:nontile}. 

Returning to the quest for tiles inside a covering convex body we deduce the following result. We call a set \emph{polytopal} if it is a finite union of convex polytopes. In this paper we use the short term ``tile'' for translative tiles with respect to $\ZZ^d$. They are not necessarily convex.

\begin{corollary}
Let $K\subset\RR^d$ be a convex body with $K+\ZZ^d = \RR^d$. Then there exists a polytopal tile $\Omega\subseteq K$.
\end{corollary}

This is a direct consequence of Theorem \ref{thm:main}, since any covering convex body contains a minimal covering body (Proposition \ref{prop:zorn}), which is a covering polytope. Inside a covering polytope one constructs a polytopal tile by considering the intersections with a fixed polytopal tiling of $\RR^d$, for instance by the cube $[0,1]^d$. 
This result relies heavily on the convexity of $K$. If the convexity assumption is dropped, it is easy to find non-polytopal tiles.

In order to obtain more non-trivial examples of minimal covering bodies, we revisit and generalize the example from Xue and Zong \cite{xuezong}. We obtain that the Cayley sum of minimal covering bodies $Q_1$ and $Q_2$ is again minimal covering if
it is covering.
This is equivalent requiring that $Q_1$ and $Q_2$ satisfy
\begin{equation}
\label{eq:bm_property}
(1-\lambda)Q_1 + \lambda Q_2 + \ZZ^d = \RR^d,\quad\text{for all}\quad\lambda\in [0,1],
\end{equation} 
i.e., any Minkowski convex combination of $Q_1$ and $Q_2$ is a covering body (not necessarily minimal) (cf.\ Propositions \ref{prop:wiggly_sums} and \ref{prop:tensor}).  Perhaps surprisingly, not all pairs $Q_1$ and $Q_2$ will have this property: The parallelograms
\begin{equation}
\label{eq:parallelograms}
\begin{split}
P_1 &= \{ \alpha (1,1) + \beta (-\tfrac 12, \tfrac 12) : 0\leq \alpha,\beta\leq 1\}\quad \text{and}\\
P_2 &= \{ \alpha (-1,1) + \beta (\tfrac 12, \tfrac 12) : 0\leq \alpha,\beta\leq 1\}
\end{split}
\end{equation}
are both covering, but $\tfrac 12 P_1 + \tfrac 12 P_2$ is not, see Figure \ref{fig:cov_example}. A sufficient criterion in the plane is that the lattice vectors in the difference body $Q_1-Q_1$ are contained in $Q_2-Q_2$. Here, $A-B$ denotes the set of vectors $a-b$, where $a\in A$ and $b\in B$. 

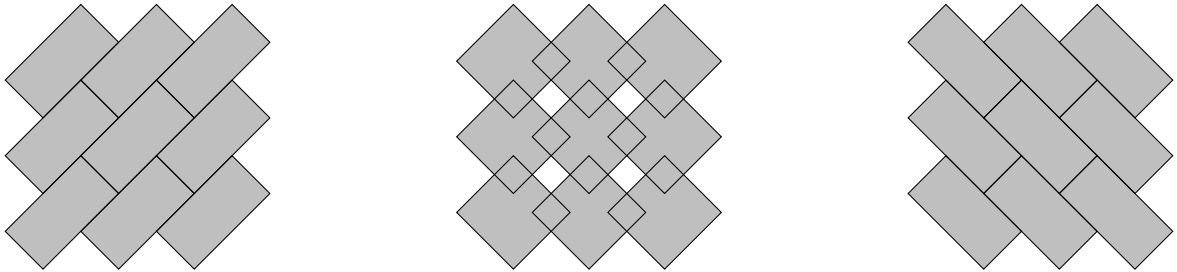
\begin{figure}[!h]
\centering

\begin{minipage}{0.25\textwidth}
\centering
\begin{tikzpicture}
\foreach \x in {-1,0,1}
	\foreach \y in {-1,0,1}{
		\fill[lightgray] (\x,\y) --(\x+1,\y+1) --(\x+0.5,\y+1.5) --(\x-0.5,\y+0.5) -- cycle;
		%\fill[black] (\x, \y) circle(2pt);
	}
\foreach \x in {-1,0,1}
	\foreach \y in {-1,0,1}
		\draw  (\x,\y) --(\x+1,\y+1) --(\x+0.5,\y+1.5) --(\x-0.5,\y+0.5) -- cycle;
\end{tikzpicture}
\end{minipage}
\hfill
\begin{minipage}{0.25\textwidth}
\centering
\begin{tikzpicture}
\foreach \x in {-1,0,1}
	\foreach \y in {-1,0,1}{
		\fill[lightgray] (-0.75+\x,0.75+\y) --(\x,\y) --(0.75+\x,0.75+\y) --(\x,1.5+\y) -- cycle;
		%\fill[black] (\x, \y) circle(2pt);
	}
\foreach \x in {-1,0,1}
	\foreach \y in {-1,0,1}
		\draw (-0.75+\x,0.75+\y) --(\x,\y) --(0.75+\x,0.75+\y) --(\x,1.5+\y) -- cycle;
\end{tikzpicture}
\end{minipage}
\hfill
\begin{minipage}{0.25\textwidth}
\centering
\begin{tikzpicture}
\foreach \x in {-1,0,1}
	\foreach \y in {-1,0,1}{
		\fill[lightgray] (\x,\y) --(\x+0.5,\y+0.5) --(\x-0.5,\y+1.5) --(\x-1,\y+1) -- cycle;
		%\fill[black] (\x, \y) circle(2pt);
	}
\foreach \x in {-1,0,1}
	\foreach \y in {-1,0,1}
		\draw (\x,\y) --(\x+0.5,\y+0.5) --(\x-0.5,\y+1.5) --(\x-1,\y+1) -- cycle;
\end{tikzpicture}
\end{minipage}
\caption{Lattice arrangements for $P_1$, $\frac{1}{2}P_1+\frac{1}{2}P_2$ and $P_2$.}
\label{fig:cov_example}
\end{figure}

\begin{theorem}[Theorem \ref{thm:tile_covering_text}]
\label{thm:tile_covering}
Let $Q_1,Q_2\subset\RR^2$ be convex tiles with $(Q_1-Q_1)\cap\ZZ^2 \subseteq (Q_2-Q_2) \cap \ZZ^2$. Then, \eqref{eq:bm_property} holds true. Moreover, for any $\lambda\in [0,1]$, $\mu\in [0,1)$, the body $\mu((1-\lambda)Q_1 + \lambda Q_2)$ is \emph{not} covering. 
\end{theorem}

The fact that \eqref{eq:bm_property} is in general false raises the natural question how well the Minkowski sum of covering convex bodies can cover the space. The classical way of measuring the covering quality of a convex body $K\subset\RR^d$ with respect to translations by $\ZZ^d$ is by considering its \emph{covering radius} with respect to the integer lattice:
\begin{equation}
\label{eq:covering_radius}
\mu(K) = \inf\{\mu\geq 0 : \mu K + \ZZ^d = \RR^d\}.
\end{equation}
This raises the following question: Given $\mu(K_1)$ and $\mu(K_2)$, how large can the covering radius of $\mu(K_1+K_2)$ be? This problem has a similar character to the classical Brunn-Minkowski inequality, which is regarded as one of the fundamental theorems in convex geometry. It states that
\begin{equation}
\label{eq:bm}
\vol(K_1 + K_2)^{1/d} \geq \vol(K_1)^{1/d} + \vol(K_2)^{1/d},
\end{equation}
where $K_1,K_2\subset\RR^d$ are non-empty convex bodies (or, more generally, compact sets) and $\vol(K)$ denotes the Lebesgue measure of $K$. We refer to \cite{gardner} for more background on the Brunn-Minkowski inequality and its far-reaching implications throughout mathematics. Beyond its similarity to the Brunn-Minkowski inequality, the covering radius of Minkowski sums has been studied in the context of the Lonely Runner Conjecture \cite{shifted_lr,blancocriadosantos,covrad_lonely_runner}.

Since the covering radius is $(-1)$-homogeneous, we aim for an inequality of the form
\begin{equation}
\label{eq:mu_bm_prototype}
\mu(K_1 + K_2)^{-1} \geq c_d(\mu(K_1)^{-1} +\mu(K_2)^{-1}),
\end{equation}
for some constant $c_d > 0$ depending only on the dimension $d$. If $K$ is lower-dimensional, and, thus, $\mu(K)=\infty$, we say that $\mu(K)^{-1} = 0$. The example from \eqref{eq:parallelograms} shows that we cannot expect $c_d=1$ for $d > 1$. At the same time, one trivially has by monotonicity that $c_d \geq \tfrac 12$. We provide the optimal value for $c_2$.

\begin{theorem}[Theorem \ref{thm:planar_mu_bm_text}]
\label{thm:planar_mu_bm}
Let $K_1,K_2\subset\RR^2$ be convex bodies. Then,
\[
\mu(K_1 + K_2)^{-1} \geq \tfrac 34 (\mu(K_1)^{-1} + \mu(K_2)^{-1}).
\]
Equality holds, e.g., for the parallelograms $P_1$ and $P_2$ from \eqref{eq:parallelograms}.
\end{theorem}

In the language of the covering radius, Theorem \ref{thm:tile_covering} states that $\mu((1-\lambda)Q_1 + \lambda Q_2)=1$ for any convex tiles satisfying $(Q_1-Q_1)\cap\ZZ^2 \subseteq (Q_2-Q_2) \cap \ZZ^2$. This assumption can be enforced if we are free to act with a $\GL_2(\ZZ)$ transformation on the individual tiles. Here, $\GL_2(\ZZ)$ denotes the unimodular linear group, i.e., the group of linear matrices $g\in\ZZ^{2\times 2}$ for which $g^{-1}$ exists in $\ZZ^{2\times 2}$.

\begin{corollary}
\label{cor:isomorphic_mu_bm}
Let $K_1,K_2\subset\RR^2$ be convex bodies. Then there exists a unimodular transformation $g\in \GL_2(\ZZ)$ such that
\[
\mu((1-\lambda)K_1 + g\cdot\lambda K_2)^{-1} \geq (1-\lambda)\mu(K_1)^{-1} + \lambda\mu(K_2)^{-1}.
\]
\end{corollary}

The corollary is obtained by comparing $K_1$ and $K_2$ to the dilations of convex tiles they contain, choosing  an appropriate $g\in \GL_2(\ZZ)$ so that after applying it to the second tile the two tiles satisfy the conditions of Theorem \ref{thm:tile_covering}. The existence of such a unimodular transformation is explained in Remark \ref{rem:tiles}. The corollary then follows by Theorem \ref{thm:tile_covering}. For general dimension $d$, by the monotonicity of the covering radius we see that (dilations of) minimal covering bodies will be extremal in inequalities of the form \eqref{eq:mu_bm_prototype}, so it seems promising to investigate covering bodies further in the context of these Brunn-Minkowski type inequalities.

\subsection*{Related concepts}

Minimal covering bodies are an instance of extremal bodies with respect to a given geometric functional. These extremal bodies appear in different mathematical contexts and are key ingredients in various geometric questions. Here we try to give a review of these related notions.
In order to describe this in full generality, let $\K^d$ denote the set of convex bodies in $\RR^d$ and consider a weakly monotonous functional $\varphi\colon\K^d\to\RR$.
A convex body is called $\varphi$-minimal (resp.\ maximal) if it is a minimal (maximal) element of $\varphi^{-1} ( \{\varphi(K)\} )$, ordered by inclusion.

Classical examples from Euclidean geometry are the reduced convex bodies, which are width-minimal, and the complete bodies (more commonly known as bodies of constant width), which are diameter-maximal, see  for instance \cite{lassakmartini, martini}. In the geometry of numbers, analogous notions have been studied in \cite{barany, coolslemmens} for lattice polygons and in \cite{codenottifreyer} for general convex bodies. A convex body is called lattice reduced, if it is lattice width minimal and it is called lattice complete if it is maximal with respect to the functional $K\mapsto \lambda_1(K-K)^{-1}$, where $\lambda_1$ denotes the first successive minimum, which resembles a continuous lattice version of the diameter.

Let us consider the covering radius as a functional
\(
\mu \colon \K^d\to\RR\cup\{\infty\}.
\)
Recall that $\mu(K) < \infty$ if and only if $K$ is full dimensional. It is easy to see that minimal covering bodies are exactly the $\mu$-minimal bodies with $\mu(K) = 1$, and, by homogeneity, any full dimensional $\mu$-minimal body is a dilation of a minimal covering body.

The covering radius is one of the functionals that admits both minimal and maximal bodies. The $\mu$-maximal bodies appear as ``tight'' bodies in \cite{codenottischymurasantos} and are closely related to ``inclusion maximal hollow'' bodies, i.e., convex bodies that are inclusion maximal under the condition of containing no interior lattice points. Both $\mu$-maximal and lattice width minimal bodies play a critical role in the study of flatness type problems as they arise in geometry of numbers and integer programming, as well as in toric geometry \cite{averkovcodenottifreyerhuang, codenottifreyer, coolslemmens}.

Finally, similar to covering, one might ask for the structure of maximal packing convex bodies. A convex body is called \emph{packing} if the translates $(K-z)_{z\in\ZZ^d}$ are non-overlapping, and it is called \emph{maximal packing} if there exists no convex body $K' \supsetneq K$ which is also packing. Since $\tfrac 12\lambda_1(K-K)$ describes the maximal dilation of $K$ so that it is packing, we see that the packing maximal bodies correspond, up to dilations, to the lattice complete bodies from \cite{codenottifreyer}. It is interesting to note that, unlike for covering minimal bodies, there exist non-tiling packing maximal bodies already in the plane \cite[Proposition 3.12]{codenottifreyer}. 

\subsection*{Outline} In Section \ref{sec:mumin} we will prove Theorem \ref{thm:polytopes}. This is done by considering a local characterization of minimal covering bodies. We will, moreover, generalize the example from \cite{xuezong} of a non-tiling minimal covering body and provide a new family of non-trivial examples for minimal covering bodies, namely the Cayley sums of convex tiles as in Theorem \ref{thm:tile_covering}.

In Section \ref{sec:bm} we prove Theorem \ref{thm:planar_mu_bm}. The argument uses an analysis of the arithmetic structure of the lattice polygons contained in convex tiles, together with a ``Cayley trick''. We also make some observations on \eqref{eq:mu_bm_prototype} in higher dimensions.

\subsection*{Acknowledgements} 
We thank Francisco Santos and Matthias Schymura for showing us the parallelograms in \eqref{eq:parallelograms}. We thank Yanlu Lian and Fei Xue for getting in touch and telling us about their research parallel to ours, and Mei Han for establishing this exchange.

\section{Minimal covering bodies}
\label{sec:mumin}

%ball lemma

We begin by showing the existence of minimal covering bodies within covering bodies by using Zorn's lemma.

\begin{proposition}
\label{prop:zorn}
    Let $K\subset\RR^d$ be a covering convex body. Then, $K$ contains a minimal covering body.
\end{proposition}

\begin{proof}
   We want to use Zorn's lemma to show that 
    \[
        \mathcal{P} = \{ L\in\K^d \colon L\subseteq K,~L+\ZZ^d = \RR^d\},  
    \]
    partially ordered by inclusion,
    contains a minimal element, which is then the desired minimal covering body. To this end, let $\mathcal{C}\subseteq\mathcal{P}$ be a chain. We show that $M=\bigcap\mathcal{C}$ is a covering body. Clearly, $M$ is convex and compact, so it suffices to show that $M$ is covering. To this end, let $p\in\RR^d$ be arbitrary. For any $L\in \mathcal{C}$, it holds that $L+\ZZ^d = \RR^d$ and $M\subset L$,  there exists a non-empty subset $I_L\subset \{1,\dots,n\}$ such that $p\in x_i+L$, for all $i\in I_L$. Let $I=\bigcap_{L\in\mathcal{C}} I_L$. $I$ is non-empty; otherwise, since $\mathcal C$ is a chain, there was a $L$ such that $I_L=\emptyset$, a contradiction. Hence, there exists $i\in I$ such that $p\in x_i + L$, for all $L\in\mathcal C$. Thus, $p\in x_i +M$. Since $p$ was arbitrary, we have $M+\ZZ^d = \RR^d$.
\end{proof}

Our first lemma gives a local criterion for a minimal covering body. In order to state it, we recall that a point $v$ in a convex body $K\subset\RR^d$ is called \emph{exposed}, if there exists a normal $u\in\RR^d$ such that $\langle v,u\rangle > \langle x, u\rangle$ holds for all $x\in K\setminus\{v\}$. The vector $u$ is said to expose $v$ and the set of exposed points of $K$ is denoted by $\expo(K)$. It is known that $K$ is the closure of the convex hull of its exposed points \cite[Theorem 1.4.7]{schneider}. In particular, $K$ is a polytope if and only if it has finitely many exposed points. In this case, we refer to the exposed points as the \emph{vertices} of $K$.

\begin{lemma}
\label{lemma:balls_lemma} Let $K\subset \RR^d$ be a convex body. Then, the following statements are equivalent:
\begin{enumerate}
\item $K$ is minimal covering.
\item $K$ is covering and for all $v\in \expo(K)$, $u\in \RR^d$ exposing $v$ in $K$ and all $\varepsilon>0$, the convex body
$$K_{u, \varepsilon}:=K\cap \{x\in \RR^d \; | \; \langle x,u\rangle \leq \langle v,u\rangle-\varepsilon \}$$
is not covering.
\item For all $v\in \RR^d$ and $0<\varepsilon<\frac{1}{2}$, 
\begin{equation}\label{eq:balls_covering}
\bigcup \limits_{z\in \ZZ^d} (K-z) \cap \ball (v, \varepsilon)=\ball(v,\varepsilon),
\end{equation}
and, if $v\in\expo(K)$,
\begin{equation}\label{eq:balls_condition}
\ball(v,\varepsilon)\cap K \not\subseteq\bigcup \limits_{z\in \ZZ^d\setminus \{\origin\}} (K-z) \cap \ball (v, \varepsilon).
\end{equation}
\end{enumerate}
\end{lemma}

\begin{proof} 
We will first prove the equivalence of the first two statements.
For an arbitrary $v\in \expo(K)$, $u\in \RR^d$ exposing $v$ in $K$ and all $\varepsilon>0$, notice that $K_{u,\varepsilon}\subsetneq K$. Therefore, if $K$ is minimal covering, this convex body will not be covering.

For the converse implication, assume that $K$ is not minimal covering. Then there exists $L\subsetneq K$ such that $L+\ZZ^d=\RR^d$. Let $v$ be an exposed point in $K$ which is not in $L$ and $u\in \RR^d$ a functional exposing $v$ in $K$. For $\varepsilon:=\langle v,u\rangle-\max_{x\in L}\langle x,u\rangle>0$, notice that $L\subset K_{u,\varepsilon}$. Therefore, $\RR^d=L+\ZZ^d\subset K_{u,\varepsilon}+\ZZ^d,$ and the second statement cannot hold.

The next goal is to show that the second statement implies the third. 
Let $v\in \RR^d$ and $0<\varepsilon<\frac{1}{2}$ be arbitrary. Since $K+\ZZ^d=\RR^d$, we can intersect both sides with $\ball(v, \varepsilon)$ and writing out the Minkowski sum we obtain 
\[
\bigcup \limits_{z\in \ZZ^d} \ball (v, \varepsilon)\cap (K-z)=\ball(v,\varepsilon).
\]
 Towards a contradiction, assume that there is an exposed point $v\in \expo(K)$ and $0<\varepsilon<\frac{1}{2}$ such that
\begin{equation}\label{eq:balls_contradiction}
\ball(v,\varepsilon)\cap K \subseteq\bigcup \limits_{z\in \ZZ^d\setminus \{\origin\}} \ball (v, \varepsilon)\cap (K-z).
\end{equation}
\noindent Let $u\in \RR^d$ be a vector exposing $v$ in $K$, and set 
\[
\varepsilon':= \langle v,u\rangle-\max_{x\in K\cap \partial \ball(v,\varepsilon)}\langle x,u\rangle>0.
\]
%We claim that then $K_{u, \varepsilon'}$ will be covering.\\
Let $y\in K\setminus K_{u, \varepsilon'}$. Then, \[
\langle y,u\rangle \geq \langle v,u\rangle-\varepsilon'=\max_{x\in K\cap \partial \ball(v,\varepsilon)}\langle x,u\rangle.
\]
 Let $y'$ be the intersection of $\partial\ball(v, \varepsilon)$ and the ray starting from $v$ going through $y$. Then, $y=v+(y-v)=v+\lambda(y'-v)$, where $\lambda>0$ by definition and $\lambda\leq 1$ if and only if $y\in \ball(v, \varepsilon)$.
Calculating
\[
\begin{split}
\langle y,u\rangle &= \langle v+\lambda(y'-v) , u \rangle = (1-\lambda)\langle v,u\rangle +\lambda \langle y',u\rangle\\
& \leq (1-\lambda)\langle v,u\rangle + \lambda\max_{x\in K\cap \partial \ball(v,\varepsilon)} \langle x,u\rangle \\
&=(1-\lambda)\langle v,u\rangle+\lambda(\langle v,u\rangle-\varepsilon')=\langle v,u\rangle-\lambda\varepsilon'.
\end{split}
\]
Since $\langle y,u\rangle\geq\langle v,u\rangle - \varepsilon'$ by supposition and $\varepsilon'>0$, we conclude that $\lambda\leq 1$, that is $y\in \ball(v,\varepsilon)$. Therefore, $K\setminus K_{u,\varepsilon'}\subseteq \ball(v,\varepsilon)\cap K$.

 Let $z\in \ZZ^d\setminus \{ \origin\}$. Notice that $(K\setminus K_{u, \varepsilon'})\cap \ball(v+z, \varepsilon)\subseteq \ball(v,\varepsilon)\cap \ball(v+z,\varepsilon)$. However, we chose $\varepsilon < \frac{1}{2}$ which provides $\ball(v, \varepsilon)\cap\ball(v+z,\varepsilon)=\emptyset$. Since $K_{u,\varepsilon'}\subset K$ we can conclude that $K\cap \ball(v+z, \varepsilon)\subseteq K_{u,\varepsilon'}$. Translating by $-z$, we get \(\ball(v,\varepsilon)\cap(K-z)\subseteq K_{u,\varepsilon'}-z\).
 
Combining these two inclusions with equation \eqref{eq:balls_contradiction}, we see
\[
\begin{split}
K\setminus K_{u,\varepsilon'}&\subseteq K\cap \ball(v,\varepsilon)\subseteq \bigcup\limits_{z\in \ZZ^d\setminus\{\origin\}}\ball(v,\varepsilon)\cap(K-z)\\
&\subseteq \bigcup\limits_{z\in \ZZ^d\setminus\{\origin\}}\ball(v,\varepsilon)\cap(K_{u,\varepsilon'}+z)=\ball(v,\varepsilon)\cap (K_{u,\varepsilon'}+(\ZZ^d\setminus\{\origin\})).
\end{split}
\]
Specifically, this shows that $K\setminus K_{u, \varepsilon'}\subseteq K_{u, \varepsilon'}+\ZZ^d$, and since $\origin\in \ZZ^d$, $K\subset K_{u, \varepsilon'}+\ZZ^d$. By supposition of $K$ being covering, this shows that $K_{u, \varepsilon'}$ is covering, which is a contradiction with the second statement. 

Finally, we show that the third statement implies the first one. It follows directly from \eqref{eq:balls_covering} that $K$ is covering. To see that $K$ is minimal covering, let $L\subsetneq K$ be a convex body and $v\in \expo(K)\setminus L$. Set $\varepsilon:=\min \{ \frac{1}{4}, \frac{1}{2}d(v,L)\}>0$, where $d(v,L)$ denotes the Euclidean distance of $v$ to $L$. From the condition \eqref{eq:balls_condition} we see that there exists some $y\in \ball(v,\varepsilon)\cap K$ which is not in $\ball(v,\varepsilon)\cap (K-z)$ for any $z\in \ZZ^d\setminus\{\origin\}$. Since $L\subseteq K$, this implies that $y\notin L-z$ for all $z\in \ZZ^d\setminus \{\origin\}$. Because $\varepsilon\leq\frac{1}{2}d(v,L)$, $K\cap \ball(v,\varepsilon)\subseteq K\setminus L$. Since $y\in K\cap \ball(v,\varepsilon)$, this shows that $y\notin L$. Therefore, $L$ is not covering.
\end{proof}

We observe that, although in equations \eqref{eq:balls_covering} and \eqref{eq:balls_condition} we wrote infinite unions, in fact there are finitely many $z \in \ZZ^d$ such that $v \in K-z$, and therefore, for small $\varepsilon$, finitely many non-empty intersections of the form $\ball(v, \varepsilon)\cap K-z$. These special lattice points $z$ will play an important role in what follows, and in particular in the proof of polytopality for minimal covering bodies, which justifies giving them a name. 

For a convex body $K\subset \RR^d$ and $v\in \expo(K)$, we consider the set of \textit{feasible lattice vectors} of $v$ in $K$ to be 
\[
\orbit(K,v)=(K-v)\cap \ZZ^d.
\]
Equivalently, the feasible lattice vectors are those vectors which produce translations $K-z$ of the convex body which contain the point $v$.

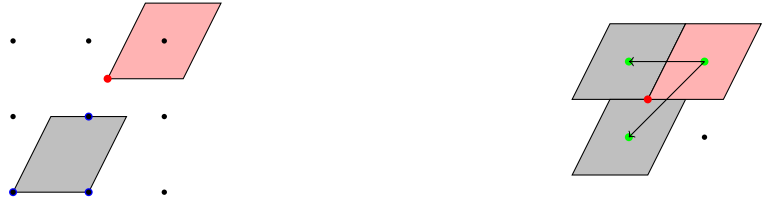
\begin{figure}[h]
\centering
\begin{minipage}{0.45\textwidth}
\centering
\begin{tikzpicture}
\filldraw[fill=red!30, draw=black] (3/4,1/2)--(-1/4,1/2)--(-3/4,-1/2)--(1/4,-1/2)--cycle;
\filldraw[fill=lightgray, draw=black] (0-1/2,0-1)--(-1-1/2,0-1)--(-3/2-1/2,-1-1)--(-1/2-1/2,-1-1)--cycle;
%\fill[red!30]  (3/4,1/2)--(-1/4,1/2)--(-3/4,-1/2)--(1/4,-1/2)--cycle;
\fill[red] (-3/4,-1/2) circle(1.5pt);
%\fill[lightgray]  (0-1/2,0-1)--(-1-1/2,0-1)--(-3/2-1/2,-1-1)--(-1/2-1/2,-1-1)--cycle;
\fill[blue] (-2,-2) circle(1.5pt);
\fill[blue] (-1,-2) circle(1.5pt);
\fill[blue] (-1,-1) circle(1.5pt);
\foreach \x in {-2,...,0}
	\foreach \y in {-2,...,0}
		\fill[black] (\x,\y) circle(1pt);

\end{tikzpicture}
\end{minipage}
\begin{minipage}{0.45\textwidth}
\centering
\begin{tikzpicture}
\filldraw[fill=lightgray, draw=black] (3/4-1,1/2)--(-1/4-1,1/2)--(-3/4-1,-1/2)--(1/4-1,-1/2)--cycle;
\filldraw[fill=lightgray, draw=black] (3/4-1,1/2-1)--(-1/4-1,1/2-1)--(-3/4-1,-1/2-1)--(1/4-1,-1/2-1)--cycle;

\filldraw[fill=red!30, draw=black] (3/4,1/2)--(-1/4,1/2)--(-3/4,-1/2)--(1/4,-1/2)--cycle;
%\draw[thick] (3/4-1,1/2)--(-1/4-1,1/2)--(-3/4-1,-1/2)--(1/4-1,-1/2)--cycle;
%\draw[thick] (3/4-1,1/2-1)--(-1/4-1,1/2-1)--(-3/4-1,-1/2-1)--(1/4-1,-1/2-1)--cycle;
%\fill[red!30]  (3/4,1/2)--(-1/4,1/2)--(-3/4,-1/2)--(1/4,-1/2)--cycle;

\fill[red] (-3/4,-1/2) circle(1.5pt);
\foreach \x in {-1,...,0}
	\foreach \y in {-1,...,0}
		\fill[black] (\x,\y) circle(1pt);

\fill[green] (0,0) circle(1.5pt);
\fill[green] (-1,0) circle(1.5pt);
\fill[green] (-1,-1) circle(1.5pt);

\draw[->] (0,0)--(-1,0);
\draw[->] (0,0)--(-1,-1);

\end{tikzpicture}
\end{minipage}
\caption{Left: In red is the convex body $K$ and an exposed point $v$ of $K$; in gray we see $K-v$ and in blue are the feasible lattice vectors. Right: We see the feasible lattice vectors as those translates of $K$ which contain $v$.}

\end{figure}

The key ingredient in the proof of polytopality is a lemma regarding how the sets of feasible lattice vectors of exposed points relate to each other.

\begin{lemma}
	\label{lem:orbits}
	Let $K\subset\RR^d$ be a minimal convex body and $v\in \expo(K)$. Then 
		\[
		\bigcap_{z \in \orbit(K,v)} (K-z) = \{v\}.
		\]
	In particular, for any $w \in \expo(K)$, $w\neq v$, we have $\orbit(K,v)\not\subseteq \orbit(K,w)$.	
\end{lemma}

\begin{proof}
	First we show that if the first claim of the lemma is true, then indeed we cannot have  $\orbit(K,w) \subseteq \orbit(K,v)$ for any exposed point $w \neq v$.
	
	Indeed, if this were to hold we would have a contradiction in
	\[
	\{w\} = \bigcap_{z \in \orbit(K,w)} (K-z) \supseteq \bigcap_{z \in \orbit(K,v)} (K-z) = \{v\}.
	\]
	
	We now want to prove the first statement. We proceed by contradiction and suppose that $\bigcap_{z \in \orbit(K,v)} (K-z) \neq \{v\}$. Since $v \in K-z$ for all $z \in \orbit(K,v)$, this means that there exists $y \neq v$ with $y \in \bigcap_{z \in \orbit(K,v)} (K-z)$. Consider now a hyperplane $H$ exposing $v$ in $K$, that is, $v \in H$ and $K \setminus \{v\} \subset H^- \setminus H$, where $H^{-}$ is the respective halfspace defined by $H$.

	Let $\varepsilon \in (0,\frac12)$ and consider the ball $\ball(v,\varepsilon)$ centered at $v$ with radius $\varepsilon$. Since $\bigcap_{z \in \orbit(K,v)} (K-z)$ is convex, there exists a point $x$ in the segment $[v,y]$ such that $x \in \ball(v, \varepsilon) \cap \bigcap_{z \in \orbit(K,v)} (K-z)$ and $x \neq v$. In particular, since $x \in K \setminus \{v\}$, we have that $x \in H^- \setminus H$.
	
	We denote $E = \ball(v, \varepsilon) \cap H$ while $C = \conv(x, E)$ is the (circular) pyramid with apex $x$ and base $E$. We can choose $\delta >0 $ small enough so that $\ball(v, \delta) \cap H^- \subset C$, for example by choosing $\delta = \dist(v, \partial C \setminus E)$.
	We will next show that
	\begin{equation}
	\label{eq:cone_covering}
	\ball(v, \delta)\cap H^- \subseteq \bigcap_{z \in \orbit(K,v) \setminus \{0\}} (K-z),
	\end{equation}
	which negates condition \eqref{eq:balls_condition} for $\delta$, since  $\ball(v, \delta)\cap K \subset \ball(v, \delta)\cap H^-$. We would therefore obtain a contradiction to the assumption that $K$ is minimal covering and thus conclude the proof.
	
	To show \eqref{eq:cone_covering}, since $\ball(v, \delta) \cap H^- \subset C$, it is enough to show
	\[
	C\subseteq \bigcap_{z \in \orbit(K,v) \setminus \{0\}} (K-z).
	\]

	Let $w \in C$. If $w \in [x,v]$, by construction we have $[x,v] \subset \bigcap_{z \in \orbit(K,v)} (K-z) \subset \bigcup_{z \in \orbit(K,v) \setminus \{0\}}$. Now if $w \notin [x,v]$, consider $\bar{w}$ the intersection of the ray from $x$ to $w$ with the disk $E$. By condition \eqref{eq:balls_covering}, we know that $\bar{w} \in \bigcap_{z \in \orbit(K,v)} (K-z)$ but since $\bar{w} \in H$ while $K \setminus \{v\} \subset H^- \setminus H$, we actually can see that $\bar{w}\in \bigcap_{z \in \orbit(K,v) \setminus \{0\}} (K-z)$. Thus there exists $z \in \orbit(K, v) \setminus \{0\}$ such that $\bar{w} \in K-z$. Since we also know $x \in K-z$, by convexity we have $w \in [x, \bar{w}] \subset K-z$, and we conclude.
\end{proof}

\begin{theorem}
\label{thm:polytopes}
	Let $K\subset\RR^d$ be a minimal covering body. Then, $K$ is a polytope.
\end{theorem}

\begin{proof}
	To show that $K$ is a polytope, it is enough to prove that the set of exposed points $\expo(K)$ is finite. 
	For any $v \in \expo(K)$, its set of feasible lattice vectors $\orbit(K,v)$ is contained in $(K-K) \cap \ZZ^d$. We will exploit the fact that $(K-K) \cap \ZZ^d$ is a finite set, and therefore there are finitely many possible sets of feasible lattice vectors.
	Consider the map 
	
	\[
	\mathcal{Z}: \expo(K) \longrightarrow 2^{(K-K) \cap \ZZ^d}
	\]
	associating to each exposed point of $K$ the set of its feasible lattice vectors. Lemma \ref{lem:orbits} guarantees, in particular, that this map is injective. Since the codomain is finite, injectivity implies that $\expo(K)$ is also finite.
\end{proof}

After having shown that minimal covering bodies are polytopes, we can simplify the third condition in Lemma \ref{lemma:balls_lemma} by only considering infinitesimally small $\varepsilon$. To make this precise, we consider for any point $v$ in a polytope $P$ the \emph{cone of feasible directions} of $P$ at $v$, which is defined as
\begin{equation}
\label{eq:fcone}
\fcone(P,v) = \{\lambda(y-v) \;\vert\; \lambda \geq 0,~y\in P\}.
\end{equation}
It consist of those directions that ``point inside $P$ from $v$''. It is easy to check that
\begin{equation}
\label{eq:fcone_identities}
\begin{split}
\fcone(P,v) &= \{\lambda(y-v) \;\vert\; \lambda \geq 0,~y\in P \cap \ball(v,\varepsilon)\}\\ 
&= \{x\in\RR^d \; \vert\; \langle x,a_i(v)\rangle \leq 0,~1\leq i \leq m\},
\end{split}
\end{equation}
for all $\varepsilon> 0$, where $a_1(v),\dots,a_m(v)$ are the outer normals of facets of $P$ that contain $v$.
In particular, $\fcone(P,v)$ is closed. This would not be the case if one applied the same definition to an arbitrary convex body, which is why we had to work with the setting of Lemma \ref{lemma:balls_lemma} in order to obtain Theorem \ref{thm:polytopes}.
For $v\not\in P$ one typically sets $\fcone(P,v) = \{\origin\}$.  

\begin{corollary}\label{cor:fcone}
	Let $K\subset\RR^d$ be a convex body. Then, the following statements are equivalent:
	\begin{enumerate}
	\item $K$ is a minimal covering body.
	\item $K$ is a polytope and for all $v\in\RR^d$, we have
	\begin{equation}
	\label{eq:fcone_covering}
	\bigcup_{z\in\ZZ^d} \fcone(K, v-z) = \RR^d
	\end{equation}
	and, if $v\in\expo(K)$,
	\begin{equation}
	\label{eq:fcone_condition}
	\fcone(K,v)\not\subseteq\bigcup_{z\in\ZZ^d\setminus\{\origin\}} \fcone(K,v-z).
	\end{equation}
	\end{enumerate}
\end{corollary}

\begin{proof}
If $K$ is a minimal covering body, then it is a polytope by Theorem \ref{thm:polytopes}. Now Equations \eqref{eq:balls_covering} and \eqref{eq:fcone_identities} imply Equation \eqref{eq:fcone_covering} by translating by $-v$ and taking the positive hull. Since for $\varepsilon>0$ small enough we have 
\begin{equation}
\label{eq:locality}
\ball(\origin,\varepsilon) \cap (K - v + z) = \ball(\origin,\varepsilon)\cap \fcone(K,v-z),
\end{equation}
 for all $z\in\ZZ^d$ and only finitely many integer translates of $K$ contain $v$, we can deduce \eqref{eq:fcone_condition} from \eqref{eq:balls_condition}.

For the converse, we observe that \eqref{eq:fcone_covering} implies that $K$ is covering. Otherwise, the left hand side would be $\{\origin\}$ for the point $v\in\RR^d$ not covered by the integer translates of $K$. Proving that $K$ is minimal covering can be done along the same lines as in the proof of Lemma \ref{lemma:balls_lemma} using \eqref{eq:locality}.
\end{proof}

Having established that minimal covering bodies are necessarily polytopes, we will continue to work with minimal covering polytopes in the following. For $X\subseteq\RR^d$, we write $\relint(X)$ for the interior of $X$ relative to the subspace topology of its affine hull.

\begin{corollary}\label{rem: no interior pts}
Let $P$ be a minimal covering polytope and let $v$ be a vertex of $P$. Then, $\orbit(P,v)\cap\relint(F) = \emptyset$ for any face $F\subset P$ with $v\in F$.
\end{corollary}

\begin{proof}
Let $F$ be a face of $P$ that contains $v$ and, towards a contradiction, let $z\in\ZZ^d$ such that $v-z\in \relint(F)$. Then, $z\neq \origin$ and $\fcone(P,v) \subset \fcone(P,v-z)$ (cf.\ \eqref{eq:fcone_identities}). This contradicts \eqref{eq:fcone_condition}.
\end{proof}

We come to the lower bound theorem.

\begin{theorem}
\label{thm:lower_bound}
Let $P\subset\RR^d$ be a minimal covering polytope. Then, $P$ has at least $2d$ facets. Moreover, this bound is sharp.
\end{theorem}

\begin{proof}
Consider a vertex $v\in P$ and the corresponding set of feasible lattice vectors $\orbit(P,v) = \{z_0,\dots,z_n\}$, where we label the vectors so that $z_0 = \origin$. It follows from \eqref{eq:fcone_condition} that there exists a vector $w\in \bigcap_{i=1}^n \fcone(P,v-z_i)^c$. In particular, $w\neq \origin$.

Recall that for any $p\in P$, the facets of the cone $\fcone(P,p)$ are determined by the facets of $P$ that contain $p$. That is, if $P = \{ x\in\RR^d : \langle x,a\rangle\leq b_a,~a\in\mathcal A\}$, where $\mathcal A\subset\RR^d$, is a non-redundant description of $P$ and $p$ satisfies $\langle a,p\rangle = b_a$, then the equation $\langle a,x\rangle = 0$ defines a facet of $\fcone(P,p)$ and $a$ is an outer normal of that facet.

In our situation, we have that $w$ is separated from every $\fcone(P,v-z_i)$, $1\leq i \leq n$, by one of its facet hyperplanes. I.e., for any $i\in\{1,\dots,n\}$ there exists a facet normal $a_i\in\mathcal A$ of $P$ and of $\fcone(P,v-z_i)$ such that $\langle a_i,w\rangle > 0$. Hence, the open cone $C = \{x\in\RR^d : \langle a_i,x\rangle > 0,~1\leq i \leq n\}$ is non-empty and we have by \eqref{eq:fcone_covering} that
\[
C \subset \bigcap_{i=1}^n \fcone(P,v-z_i)^c \subset \fcone(P,v).
\]
Since $\fcone(P,v)$ is closed, it follows that \[
\cl(C) = \{ x\in\RR^d : \langle a_i,x\rangle \geq 0,~1\leq i \leq n\}\subseteq \fcone(P,v).
\]
Since $v$ is a vertex of $P$, the cone $\fcone(P,v)$ is pointed. Therefore, among the vectors $a_i$ there are at least $d$ distinct vectors, since otherwise $\cl(C)$ would contain a line.

In order to finish the proof it suffices to observe that $v$ is not contained in any of the facets $F_i$ of $P$ defined by the vectors $a_i$. If this was the case for some $i$, it would follow that
\[
\{x\in\RR^d : \langle a_i,x\rangle \leq 0 \}\supset \fcone(P,v) \ni w,
\]
contradicting the choice of $a_i$. Since $v$ is contained in at least $d$ facets, the claimed inequality follows.

In order to see that the inequality is sharp, it suffices to recall that the cube $[0,1]^d$ is a minimal covering polytope.
\end{proof}

%polytopality
%tcone lemma
%vertex-orbit correspondence
%lower bound theorem

\subsection{A construction for non-tiling minimal bodies}
\label{sec:nontile}

In this section we generalize the example from \cite{xuezong} in order to obtain a rich class of minimal polytopes that are not convex tiles.
Let $P\subset\RR^{d_1}$ be a polytope with vertices $\{v_1,\dots,v_n\}$. For each vertex of $P$ we pick a polytope $Q_i\subset\RR^{d_2}$ and we form
\begin{equation}
\label{eq:wiggly_sums}
P[Q_1,\dots,Q_n] = \conv\left( \bigcup_{i=1}^n \{v_i\}\times Q_i\right) \subset\RR^{d_1+d_2}.
\end{equation}
In \cite{xuezong}, Xue and Zong consider the special case of this construction where $P$ is a tiling hexagon in the plane and $Q_1,\dots,Q_6$ are unit segments. This way it results in a ``perturbed'' prism with hexagonal base.
%Constructions of mu minimal bodies
\begin{proposition}
\label{prop:wiggly_sums}
	In the above setting, if the polytopes $P,Q_1,\dots,Q_m$ are minimal covering and $P[Q_1,\dots,Q_n]$ is covering, then $P[Q_1,\dots,Q_n]$ is minimal covering.
\end{proposition}

\begin{proof}
We abbreviate $\overline P = P[Q_1,\dots,Q_n]$.
A typical vertex of $\overline P$ is of the form $(v_i,w)$, where $w$ is some vertex of $Q_i$. We want to verify \eqref{eq:fcone_condition} for $(v_i,w)$, that is, we want to find a vector $(y_1,y_2)\in\fcone(\overline P,(v_i,w))$ that is not contained in any of the cones $\fcone(\overline P,(v_i,w)-z)$, where $z\in\ZZ^{d_1+d_2}\setminus\{\origin\}$.

Since $P$ is a minimal covering polytope, there exists a vector $y_1\in\fcone(P,v_i)$ which is not contained in any of the cones $\fcone(P,v_i-z_1)$, where $z_1\in\ZZ^{d_1}\setminus\{\origin\}$.
Let $\pi \colon\RR^{d_1+d_2} \to \RR^{d_1}$ be the canonical projection on the first $d_1$ coordinates. Then one has for any $z=(z_1,z_2)\in\ZZ^{d_1+d_2}$ that
\[\begin{split}
\pi\,\fcone(\overline P,(v_i,w)-z) &= \fcone(\pi \overline P, \pi((v_i,w)-z))\\
& = \fcone(P, v_i - z_1).
\end{split}\]
Hence, for any vector $y_2\in\RR^{d_2}$ and any $z_1\in\ZZ^{d_1}\wozero$, $z_2\in\ZZ^{d_2}$ we have $(y_1,y_2)\not\in\fcone(\overline P, (v_i,w) - (z_1,z_2))$. It remains to choose the vector $y_2$ in such a way that $(y_1,y_2)\not\in \fcone(\overline P, (v_i,w) - (\origin,z_2))$ for any $z_2\in\ZZ^{d_2}\wozero$.

Since $Q_i$ is a minimal covering body, there exists a vector $y\in\RR^{d_2}\wozero$ such that $y\not\in \fcone(Q_i,w-z_2)$ for any $z_2\in\ZZ^{d_2}\wozero$. Since 
\[
\fcone(\overline P, (v_i,w) - (\origin,z_2)) \cap (\{\origin\}\times \RR^{d_2}) = \fcone(Q_i,w-z_2)
\]
holds, it follows that
\(
(\origin,y)\not\in\fcone(\overline P, (v_i,w) - (\origin,z_2))
\) for any $z_2\in\ZZ^{d_2}\wozero$.
Since $\fcone(\overline P, (v_i,w) - (\origin,z_2))$ is closed, there exists a sufficiently large $\lambda > 0 $ with $(\lambda^{-1}y_1,y)\not\in \fcone(\overline P, (v_i,w) - (\origin,z_2))$ for any $z_2\in\ZZ^{d_2}\wozero$. Thus, choosing $y_2 = \lambda y$ yields a vector $(y_1,y_2)$ with the desired property.
\end{proof}
%Fibers over points in P being covering is a sufficient condition for the whole thing to be covering
The natural question in light of Proposition \ref{prop:wiggly_sums} is for which constellations of $P$ and $Q_1,\dots,Q_n$ the polytope $P[Q_1,\dots,Q_n]$ is covering. A sufficient condition is the following:

\begin{proposition}
\label{prop:tensor}
If $P$ is covering and $Q_1,\dots,Q_n$ are such that any Minkowski convex combination
\[
\lambda_1Q_1 + \cdots \lambda_nQ_n,\quad \lambda_i\in [0,1],~\sum_{i=1}^n\lambda_i = 1
\]
is covering. Then, $P[Q_1,\dots Q_n]$ is covering.
\end{proposition}

\begin{proof}
Again, let $\pi \colon \RR^{d_1 + d_2} \to \RR^{d_1}$ be the canonical projection and $\overline P = P[Q_1,\dots Q_n]$. As $\pi\overline P = P$ it suffices to observe that for any $p\in P$ we have that $\overline P \cap \pi^{-1}(\{p\})$ contains a Minkowski convex combination of the polytopes $Q_1\dots, Q_n$, where the coefficients are given by the representation of $p$ as a convex combination of the vertices of $P$.
\end{proof}

The condition of Proposition \ref{prop:tensor} is fulfilled in the setting of Xue and Zong \cite{xuezong}, that is, $Q_i$ being unit line segments and so the resulting polytope $P[Q_1,\dots,Q_n]$ is covering. However, whenever the segments $Q_i$ are positioned generically, $P[Q_1,\dots,Q_n]$ will have simplicial facets and, thus, will not be a tile.
This was the first example of minimal covering bodies that are not tiles. 

Conversely, if $P=[0,1]$ is the unit segment, it is easy to see that the sufficient condition from Proposition \ref{prop:tensor} is also necessary. Indeed, in this case $P[Q_1,Q_2]$ is the so-called Cayley sum of $Q_1$ and $Q_2$. The notation will be abbreviated by $[Q_1,Q_2]$. 

In order to obtain additional examples to the ones provided by the idea of Xue and Zong, we study under which circumstances all the Minkowski convex combinations of two convex tiles $Q_1$ and $Q_2$ are covering. For all such pairs, Propositions \ref{prop:wiggly_sums} and \ref{prop:tensor} will yield a minimal covering body. 

We will make use of the following lemma.

\begin{lemma}
\label{lemma:parity}
Let $K\subset\RR^d$ be a centrally symmetric covering body. Then, $\mu(K)(K-K)$ contains a pair of lattice vectors of each non-zero coset from $\ZZ^d/2\ZZ^d$. Moreover, if $K$ is a tile, there are no interior lattice vectors in $K-K$ other than the origin.
\end{lemma}

\begin{proof}
We write $K' = \mu(K)K$. Towards a contradiction, let $x\in\{0,1\}^d$ such that $(K'-K')\cap (x+2\ZZ^d) = \emptyset$. By central symmetry, we may assume that $K'-K'=2K'$ and it follows that $\tfrac 12 x$ is not contained in any integer translate of $K'$. 

For the second part it suffices to see that an interior lattice vector $z$ of $K-K$ would mean a proper overlap of $K$ and $K+z$.
\end{proof}

We will focus on the case of $Q_1$ and $Q_2$ being planar. The following description of planar convex tiles seems to be folklore.

\begin{remark}[Planar convex tiles]
\label{rem:tiles}
By considering, for instance, the angles at the vertices, one finds that planar convex tiles $K\subset\RR^2$ are centrally symmetric hexagons or quadrilaterals. If $K$ is a centrally symmetric hexagon, the tiling it produces is face-to-face and the lattice vectors in its difference body $K-K$ form a lattice hexagon such that any two consecutive vertices are a $\ZZ^2$-basis (cf.\ Lemma \ref{lemma:parity}). Such a hexagon is $\GL_2(\ZZ)$ equivalent to
\[
H_0 = \conv \{ \pm e_1,\pm e_2, \pm(e_1-e_2)\}.
\]
Moreover, each edge of $K-K$ contains a vertex of $H_0$ in its relative interior.

The hexagon $H_0$ can be regarded as ``unimodularly regular'' hexagon. That is, it has a stabilizer subgroup $G$ in $\GL_2(\ZZ)$ which is isomorphic to the dihedral group $D_6$. One checks that the cone $\{0\leq -2x\leq y\}$ is a fundamental domain of $G$.

Let the tile $K$ be circumscribed around $H_0$ in the sense that every edge of $K - K$ contains at least one lattice point from $H_0$. Since the origin is the unique interior lattice vector in $K - K$, it follows that the vertices of $K - K$ are antipodal pairs of points in the triangles $\pm T_j$, $0\leq j \leq 2$, (see also Figure \ref{fig:tiles}) where
\[
T_0 = \conv\{-e_1,-e_2, -e_1-e_2\},\quad T_1 = T_0 +2e_1, \quad T_2 = T_0+2e_2.
\]

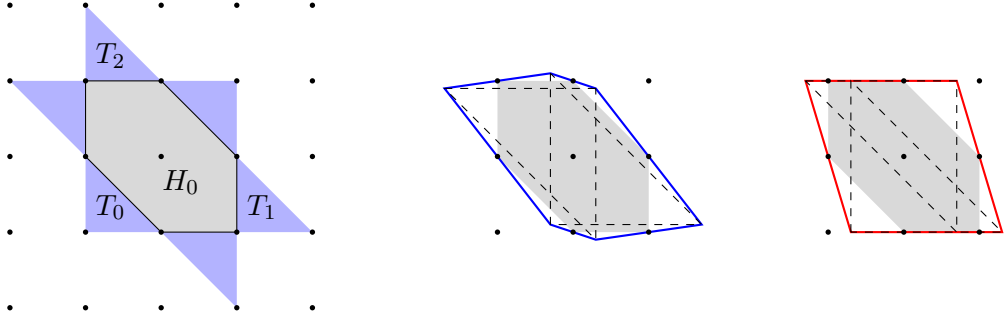
\begin{figure}[!h]
	\centering
	\begin{minipage}{0.3\textwidth}
	\centering
	\begin{tikzpicture}
		
		\fill[blue!30] (1,-1) -- (1,-2) -- (0,-1) -- cycle;
		\fill[blue!30] (0,-1) -- (-1,-1) -- (-1,0) -- cycle;
		\fill[blue!30] (-1,0) -- (-2,1) -- (-1,1) -- cycle;
		\fill[blue!30] (-1,1) -- (-1,2) -- (0,1) -- cycle;
		\fill[blue!30] (0,1) -- (1,1) -- (1,0) -- cycle;
		\fill[blue!30] (1,0) -- (2,-1) -- (1,-1) -- cycle;
		
		\fill[gray!30]
		(1,-1) -- (0,-1) -- (-1,0) --
		(-1,1) -- (0,1) -- (1,0) -- cycle;
		
		\draw
		(1,-1) -- (0,-1) -- (-1,0) --
		(-1,1) -- (0,1) -- (1,0) -- cycle;
		
		\foreach \x in {-2,-1,...,2}{
			\foreach \y in {-2,-1,...,2}{
				\fill (\x,\y) circle (1pt);
			}
		}
		\node at (.26,-.34) {$H_0$};
		\node at (-.6666,-.6666) {$T_0$};
		\node at (-.66666,1.33333) {$T_2$};
		\node at (1.333,-.6666) {$T_1$};
	\end{tikzpicture}
	\end{minipage}
	\hspace{0.05\textwidth}
	\begin{minipage}{0.25\textwidth}
	\centering
		
		\begin{tikzpicture}
			
			\fill[gray!30]
			(1,-1) -- (0,-1) -- (-1,0) --
			(-1,1) -- (0,1) -- (1,0) -- cycle;
			\draw[thick, blue] (.3,.9) -- (-.3,1.1) -- (-1.7,.9)--(-.3,-.9)-- (.3,-1.1) -- (1.7,-.9)-- cycle ;
			\draw[dashed] (.3,.9)--(-1.7,.9)--(.3,-1.1)--cycle;
			\draw[dashed] (-.3,1.1)--(-.3,-.9)--(1.7,-.9)--cycle;
			\foreach \x in {-1,...,1}{
				\foreach \y in {-1,...,1}{
					\fill (\x,\y) circle (1pt);
				}
			}
			
		\end{tikzpicture}
	\end{minipage}
	\hspace{0.05\textwidth}
	\begin{minipage}{0.25\textwidth}
			\begin{tikzpicture}
			\fill[gray!30]
			(1,-1) -- (0,-1) -- (-1,0) --
			(-1,1) -- (0,1) -- (1,0) -- cycle;

			\draw[thick,red] (-1.3, 1) -- (.7, 1) -- (1.3,-1) 
			-- (-.7, -1) --  cycle;
			
			\draw[dashed] (-1.3,1)--(.7,1)--(.7,-1)--cycle;
			
			\draw[dashed] (-.7,-1) -- (-.7,1)-- (1.3,-1) -- cycle ;
			
			\foreach \x in {-1,...,1}{
				\foreach \y in {-1,...,1}{
					\fill (\x,\y) circle (1pt);
				}
			}
			
		\end{tikzpicture}
	\end{minipage}
	\caption{Left: The hexagon $H_0$ and the regions $\pm T_i$. Center and right: examples of the difference bodies of hexagonal and quadrilateral tiles; we represent the triangles $\pm (\Delta +x)$ in dashed lines.}
	\label{fig:tiles}
\end{figure}

A quadrilateral convex tile can be viewed as a limit case of a tiling hexagon circumscribed around $H_0$, in the sense that, after a unimodular transformation, it contains $H_0$ in its difference body and has vertices on the boundary of the $T_j$ regions.

One sees that any convex tile can be expressed as $\conv( \pm( x + \Delta) ) $, for a suitable vector $x\in\RR^2$ and $\Delta = \conv\{\origin, e_1, e_2\}$ (in the case of the quadrilateral, after applying a $G$ transformation).\hfill $\diamond$
\end{remark}

\begin{theorem}
\label{thm:tile_covering_text}
Let $Q_1,Q_2\subset\RR^2$ be convex tiles with $(Q_1-Q_1)\cap\ZZ^2 \subseteq (Q_2-Q_2)\cap\ZZ^2$ and let $\lambda\in [0,1]$. Then we have
\[\mu( (1-\lambda)Q_1 + \lambda Q_2) = 1.\]
\end{theorem}

\begin{proof}
Since $(Q_1-Q_1)\cap\ZZ^2 \subseteq (Q_2-Q_2)\cap\ZZ^2$, we can consider both convex tiles $Q_1$ and $Q_2$ to be in the notation of Remark \ref{rem:tiles}, that is, $Q_i-Q_i$ is circumscribed around $H_0$.
Without loss of generality, we may assume that $Q_1$ and $Q_2$ are centered at the origin so that $Q_i - Q_i = 2Q_i$. It follows that $\tfrac 12 H_0 \subset (1-\lambda)Q_1 + \lambda Q_2$. Since $\tfrac 12 (H_0 \cup \pm T_0) = [-\frac12, \frac12]^2$ is a fundamental domain, it suffices to show that $\frac 12 T_0$ is covered by the lattice translates of $(1-\lambda)Q_1 + \lambda Q_2$.

To this end, we note that each $Q_i$ contains at least one vertex that lies in one of the triangles $\tfrac 12 T_j$ (in the hexagonal tile, there is a vertex in each triangle, but for quadrilaterals this is not the case). Since $Q_i$ is a convex tile, the translates of this vertex into the triangles $\frac 12 T_i$, $\frac 12 T_k$, $\{i,j,k\}=\{0,1,2\}$, by the respective lattice vectors are boundary points of $Q_i$. Let $v_i$ be the point in $Q_i\cap \frac 12 T_0$ that arises this way. It might be on the boundary of $\frac12 T_0$, if $Q_i$ is quadrilateral, otherwise it is in the interior.

By construction, we have $v_i + e_j \in Q_i \cap \frac12 T_j$ (where $e_0=0$). Further, we have $S_j(v_i) + e_j \subset Q_i$, where
\[
\begin{split}
S_{0}(v) &= \conv\{v, \tfrac 12 (-1,0), \tfrac 12 (0,-1)\},\\
S_{1}(v) &= \conv\{v, \tfrac 12 (-1,0), \tfrac 12 (-1,-1)\},\\
S_{2}(v) &= \conv\{v, \tfrac 12 (-1,-1),\tfrac 12 (0,-1)\}.
\end{split}
\] 
For every choice of $v\in \tfrac 12 T_0$, the three triangles $S_j(v)$ subdivide $\tfrac 12 T_0$, as shown in Figure \ref{fig:subdivT0}. Thus
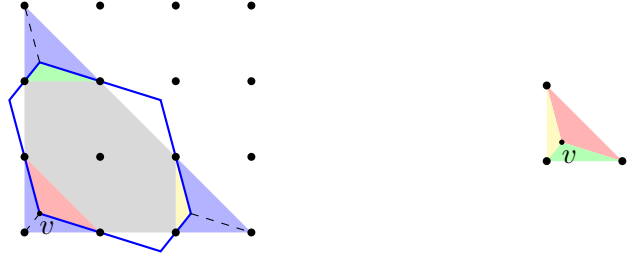
\begin{figure}
\centering
\begin{minipage}{0.45\textwidth}
\centering
		\begin{tikzpicture}
%		\fill[blue!30] (1,-1) -- (1,-2) -- (0,-1) -- cycle;
		\fill[blue!30] (0,-1) -- (-1,-1) -- (-1,0) -- cycle;
%		\fill[blue!30] (-1,0) -- (-2,1) -- (-1,1) -- cycle;
		\fill[blue!30] (-1,1) -- (-1,2) -- (0,1) -- cycle;
%		\fill[blue!30] (0,1) -- (1,1) -- (1,0) -- cycle;
		\fill[blue!30] (1,0) -- (2,-1) -- (1,-1) -- cycle;
	
	\fill[gray!30]
	(1,-1) -- (0,-1) -- (-1,0) --
	(-1,1) -- (0,1) -- (1,0) -- cycle;
	
	%			\draw[thick]
	%			(1,-1) -- (0,-1) -- (-1,0) --
	%			(-1,1) -- (0,1) -- (1,0) -- cycle;
	\fill[red!30]
	(-.8,-.75) -- (0,-1) -- (-1,0) -- cycle;
	
	\fill[green!30]
	(-.8,1.25) -- (0,1) -- (-1,1) -- cycle;	

	%\fill[green!30] (-.8,1.25-2) -- (0,1-2) -- (-1,1-2) -- cycle;	

	\fill[yellow!30]
	(1.2,-.75) -- (1,0) -- (1,-1) -- cycle;	
	
	%\fill[yellow!30](1.2-2,-.75) -- (1-2,0) -- (1-2,-1) -- cycle;	
	
	\draw[thick, blue] (.8,.75) -- (1.2,-.75) -- (.8,-1.25)--(-.8,-.75)-- (-1.2,.75) -- (-.8,1.25)-- cycle ;
	
	\draw[dashed] (-.8,-.75) -- (-1,-1);
	\draw[dashed] (-.8,1.25) -- (-1,2);
	\draw[dashed] (1.2,-.75) -- (2,-1);
	
	%the point v
	\fill (-.8,-.75) circle (1pt);
			
	%labels
	\node  at (-.7,-.94) {$v$};
	%\node [red] at (-.4,-.4) {\footnotesize $S_0(v)$};
	%\node [red] at (.6,-.7) {\footnotesize $S_1(v)+e_1$};
	%\node [red] at (-.2,.8) {\footnotesize $S_2(v)+e_2$};

	% Lattice points
	\foreach \x in {-1,...,2}{
		\foreach \y in {-1,...,2}{
			\fill (\x,\y) circle (1.5pt);
		}
	}
	
\end{tikzpicture}
\end{minipage}
\hspace{0.01\textwidth}
\begin{minipage}{0.25\textwidth}
\centering
\begin{tikzpicture}
	\fill[red!30]
	(-.8,-.75) -- (0,-1) -- (-1,0) -- cycle;
	
	\fill[green!30]
	(-.8,1.25-2) -- (0,1-2) -- (-1,1-2) -- cycle;	

	\fill[yellow!30]
	(1.2-2,-.75) -- (1-2,0) -- (1-2,-1) -- cycle;

	%the point v
	\fill (-.8,-.75) circle (1pt);
			
	%labels
	\node  at (-.7,-.94) {$v$};
	%\node [red] at (-.4,-.4) {\footnotesize $S_0(v)$};
	%\node [red] at (.6,-.7) {\footnotesize $S_1(v)+e_1$};
	%\node [red] at (-.2,.8) {\footnotesize $S_2(v)+e_2$};

	% Lattice points
	\fill (-1,-1) circle(1.5pt);
	\fill (-1,0) circle(1.5pt);
	\fill (0,-1) circle(1.5pt);
\end{tikzpicture}
\end{minipage}
	\caption{Translates of the triangles $S_j(v)$, in red, yellow and green, subdivide $T_0$}
	\label{fig:subdivT0}
\end{figure}
\[
S_j((1-\lambda)v_1 + \lambda v_2) + e_j \subseteq (1-\lambda)(S_j(v_1) + e_j) + \lambda (S_j(v)+e_j) \subset (1-\lambda)Q_1 + \lambda Q_2,
\]
which shows that $\frac 12 T_0$ is covered by appropriate translates of $(1-\lambda)Q_1 + \lambda Q_2$, and so $\mu((1-\lambda)Q_1 + \lambda Q_2) \leq 1$. 

To see that this inequality is an equality, we first observe that $(1-\lambda)v_1 + \lambda v_2 + e_j$ is on the boundary of $(1-\lambda)Q_1 + \lambda Q_2$. This is the case, since $v_i$ is on the boundary of $Q_i$ and a supporting line is given by an appropriate translation of the affine hull of the edge opposite to $v_i$ in $S_j(v_i)$. The normal vector of that line is independent of $i$. 

Next, let $z\in\ZZ^2$ be any lattice vector such that $(1-\lambda)v_1 + \lambda v_2 + z \in (1-\lambda)Q_1 + \lambda Q_2$. We have that $z\in (1-\lambda)(Q_1-Q_1) + \lambda(Q_2-Q_2)$. It follows from the representation of convex tiles as circumscribers of $H_0$ in Remark \ref{rem:tiles} that any lattice vector outside of $H_0$ is separated from $Q_i$ by a line that does not depend on $Q_i$. The separation is strict for $Q_1$, unless $Q_1$ is a tiling lattice parallelogram. In the latter case, it follows that $Q_1=Q_2$ and the assertion of the theorem becomes trivial.
Otherwise, it follows that $((1-\lambda)(Q_1-Q_1) + \lambda(Q_2-Q_2)) \cap \ZZ^d = H_0\cap\ZZ^d$ and one sees that the only lattice vectors $z$ with $(1-\lambda)v_1 + \lambda v_2 + z \in (1-\lambda)Q_1 + \lambda Q_2$ are $e_1$, $e_2$ and $\origin$.
Thus, $(1-\lambda)v_1 + \lambda v_2$ is not in the interior of any of the lattice translates of $(1-\lambda)Q_1 + \lambda Q_2$. 
\end{proof}

The following lemma is easy to see from the fact that intersecting the Cayley sum $[Q_1,Q_2]$ with the hyperplane $\{ x\in \RR^{d+1}\; : \; x_{0}=1-\lambda\}$, one obtains $\{1-\lambda\}\times ((1-\lambda)Q_1+\lambda Q_2)$. Since we will use it also in the next section, we want to record it here.

\begin{lemma}\label{lem:cayley_covering}
Let $Q_1,Q_2\subset\RR^d$ be convex bodies. Then, the Cayley sum $[Q_1,Q_2]$ is covering if and only if $Q_1$ and $Q_2$ fulfill \eqref{eq:bm_property}.
\end{lemma}

We immediately obtain:

\begin{corollary}\label{cor:planar_cayley_covering}
Let $Q_1,Q_2\subset\RR^2$ be convex tiles with $(Q_1-Q_1)\cap\ZZ^2 \subseteq (Q_2-Q_2)\cap\ZZ^2$. Then, the Cayley sum $[Q_1,Q_2]$ is covering.
\end{corollary}

\begin{remark}
The natural next question to ask is if the converse of Corollary \ref{cor:planar_cayley_covering} holds. That is, if $[Q_1,Q_2]$ is covering for convex tiles $Q_1,Q_2\subset\RR^2$, do the lattice vectors in the difference body of one tile have to be a subset of those of the other. This direction unfortunately does not hold.\\
One can check that taking two tiling paralellograms which share an edge with a lattice direction always yields a covering Cayley sum. However, two such paralellograms do not have to fulfill the condition above. For example, take $P_1:=\{ \alpha (0,1)+\beta(\frac{1}{2}, 1)\; : \; 0\leq \alpha,\beta\leq 1\}$ and $P_2:=\{ \alpha (0,1)+\beta(-\frac{1}{2}, 1)\; : \; 0\leq \alpha,\beta\leq 1\}$. 
\end{remark}

\section{Brunn-Minkowski type inequalities for the covering radius}
\label{sec:bm}

Before we come to the proof of Theorem \ref{thm:planar_mu_bm}, we give an example that shows that the covering radius of the Minkowski combination $(1-\lambda)P_1+\lambda P_2$ of two convex tiles may behave very wildly as the coefficient $\lambda$ varies. For example, consider the tiling paralellograms
\begin{align*}
	P_1:=&\{ \alpha (1,0)+\beta (\frac{1}{2},1)\; : \; 0\leq \alpha,\beta\leq 1\} \\
	P_2:=&\{ \alpha (1,1)+\beta (-\frac{3}{2},-\frac{1}{2})\; : \; 0\leq \alpha,\beta\leq 1\}. 
\end{align*}
One can see that for $\lambda$ near $0$ and near $1$, $(1-\lambda)P_1+\lambda P_2$ is covering, moreover $\mu((1-\lambda)P_1+\lambda P_2)<1$, but $\frac{1}{2}P_1+\frac{1}{2}P_2$ is not covering. In this example, $\mu((1-\lambda)P_1+\lambda P_2)^{-1}$ as a function of $\lambda$ has at least two local maxima and one local minimum, and is far from being concave. However, we will proceed to bound this function from above by a concave function.\\

For the proof of Theorem \ref{thm:planar_mu_bm} we require two elementary lemmas.

\begin{lemma}
\label{lem:triang_union}
Let $\ell\in[1/2,1]$ and let $I_1,I_2\subset\RR$ be intervals of length $\ell$ with midpoints $m_1$ and $m_2$. If
\[
1-\ell \leq |m_1-m_2|\leq \ell,
\]
then $\ZZ + (I_1\cup I_2) = \RR$.
\end{lemma}

\begin{proof}
From the upper bound it follows that $J = I_1\cup I_2$ is an interval. From the lower bound it follows that its length is at least 1. 
\end{proof}

\begin{lemma}
\label{lem:squeeze}
Let $\ell\in[1/2,1]$ and let $\beta \in \ZZ_{>0}$. There are exist exactly $2\lfloor \beta \ell\rfloor -\beta + 1$ numbers $x\in\{0,\dots,\beta-1\}$ for which there exists a $y\in\ZZ$ with 
\begin{equation}
\label{eq:squeeze}
1-\ell \leq \left|y+\tfrac x\beta\right| \leq \ell
\end{equation}
holds.
\end{lemma}

\begin{proof}
Consider $x\in\{0,\dots,\beta-1\}$ and write $\gamma_x = x/\beta$. Since $\gamma_x\in [0,1)$, the only choices for $y$ that might satisfy the inequality \eqref{eq:squeeze} are $y=0$ and $y=-1$. It follows that $x$ is admissible if and only if
\[
\gamma_x\in[1-\ell,\ell]\quad\text{or}\quad 1-\gamma_x\in [1-\ell,\ell].
\]
Since for any $\gamma\in [0,1)$ we have $1-\gamma\in [1-\ell,\ell]$ if and only if $\gamma\in [1-\ell,\ell]$, we see that $x$ is admissible if and only if $\gamma_x\in[1-\ell,\ell]$, i.e.,
\[
x\in \{0,\dots,\beta-1\} \cap [(1-\ell)\beta, \ell\beta].
\]
We can count the set on the right by counting the integers in $[0,\ell\beta]$ and $[(1-\ell)\beta,\beta)$ and substracting $\beta$. This gives claimed number of solutions.
\end{proof}

We start by proving the inequality for convex combinations of convex tiles.

\begin{lemma}\label{lem:1.5}
Let $Q_1,Q_2\subset \RR^2$ be convex tiles and $\lambda \in [0,1]$. Then, $\frac{4}{3}((1-\lambda)Q_1+\lambda Q_2)$ is covering.
\end{lemma}

\begin{proof}
%Setting up the case distincions
The problem is symmetric with respect to $Q_1$ and $Q_2$, so we assume that $(1-\lambda)\geq \tfrac{1}{2}$, that is, $\lambda \leq \tfrac{1}{2}$.
As per Remark \ref{rem:tiles}, assume that $H_0=\conv\{\pm e_1, \pm e_2, \pm(e_1-e_2) \}\subseteq Q_1-Q_1$. Let $c\geq 0$ be minimal such that $(Q_2-Q_2)\cap \ZZ^2\subset cH_0$. 
Note that since $H_0$ is a reflexive polygon, $c$ is an integer.

%On the nature of alpha and beta
First, we assume that $c\geq 6$. By the minimality of $c$, there exists a vector $v = (\alpha,\beta)\in (Q_2-Q_2)\cap\ZZ^2$ with $v\in cH_0\setminus (c-1)H_0$. We can assume without loss of generality that $0\leq -2\alpha\leq \beta = c$ (cf.\ Remark \ref{rem:tiles}). Moreover, since $Q_2$ is a tile, the vector $v$ is primitive, i.e., $\gcd(\alpha,\beta)=1$. In particular, we have $0 < -2\alpha < \beta$ and $\beta \geq 6$.

If $\lambda < \tfrac 18$, we easily conclude by using the trivial observation that $(1-\lambda)Q_1 + \lambda Q_2 \supseteq (1-\lambda)Q_1$ and therefore

\[
\mu((1-\lambda)Q_1 + \lambda Q_2)^{-1} \geq (1-\lambda)\mu(Q_1)^{-1} \geq \tfrac 78 > \tfrac 34.
\]
If instead $\lambda \geq \tfrac 18$, consider the parallelogram 
\[
P = (1-\lambda)[0,e_1] + \lambda [0,v] \subset (1-\lambda)Q_1 + \lambda Q_2.
\]
We show that $\tfrac 43 P$ is covering.
To this end, it suffices to show that the line $L_h = \{y=h\}$, $0\leq h < 1$, are covered by integer translates of $P$. 
Let $z\in\ZZ^2$ be such that $z+\tfrac 43 P$ intersects $L_h$ for any $0\leq h < 1$. Since the vertical height of $\tfrac 43 P$ is $\tfrac 43 \lambda\beta$ and since $z$ is integral it follows that
\[
\lceil 1-\tfrac 43\lambda\beta\rceil \leq z_2 \leq 0.
\]

We observe that such a $z_2 \in \ZZ$ exists exactly if $\tfrac 43 \lambda \beta \geq 1$, and since $\beta \geq 6$, this holds whenever $\lambda \geq \tfrac 18$, which we have assumed. Thus $X:= \{\lceil 1-\tfrac 43\lambda\beta\rceil,\dots,0\}$ is not empty.

Conversely, we observe that for any $z\in\ZZ^2$ with $z_2\in X$, we have $(z+\tfrac 43P)\cap L_h\neq\emptyset$ for all $0\leq h < 1$.
For any $z_2 \in X$ we see that
\[
(z+\tfrac 43 P) \cap L_h = [(z_1 + \tfrac \alpha \beta (h-z_2), h),~(z_1 + \tfrac \alpha \beta (h-z_2) + \tfrac 43(1-\lambda), h)].
\]
Forgetting the second coordinate this corresponds to an interval of length $\ell = \tfrac 43 (1-\lambda)$ and midpoint 
\[
m(z_1,z_2,h)= z_1 + \tfrac \alpha \beta (h-z_2) + \tfrac 23(1-\lambda).
\] 
Recall that by assumption $1-\lambda \geq \tfrac 12$ and, thus $\ell \geq \tfrac 23$. If $\ell = 1$, we are clearly done, so we assume $\ell\in [\tfrac 23, 1)$.

In order to show that $L_h$ is covered, we apply Lemma \ref{lem:triang_union}. To do so,  it is enough to find $z_2,\widetilde z_2\in X$ for which there exist $z_1,\widetilde z_1\in\ZZ$ with
\[
1-\ell \leq |m(z_1,z_2,h) - m(\widetilde z_1,\widetilde z_2,h)| \leq \ell.
\]
This is equivalent to finding $x\in (X-X)\,\mathrm{mod}~\beta$ for which there is $y\in\ZZ$ with
\begin{equation}
\label{eq:alpha_squeeze}
1-\ell \leq |y + \tfrac \alpha\beta x| \leq \ell.
\end{equation}
Because of $\gcd(\alpha,\beta)=1$, the map $x\mapsto \alpha x$ defines an isomorphism on $\ZZ\,\mathrm{mod}~\beta$. Hence, Lemma \ref{lem:squeeze} implies that there are exactly $2\lfloor \beta \ell\rfloor -\beta + 1 > 0$ solutions to \eqref{eq:alpha_squeeze} in $\ZZ\,\mathrm{mod}~\beta$. Since $\ell > \tfrac 23$ and $\beta \in \ZZ_{\geq 6}$, this number is greater than zero.

One sees that $(X-X)\,\mathrm{mod}~\beta$ contains $\min\{\beta,\,2\lfloor\tfrac 43 \lambda\beta\rfloor - 1\}$ elements. If $|(X-X)\,\mathrm{mod}~\beta|=\beta$, that is, $(X-X)\,\mathrm{mod}~\beta= \ZZ\,\mathrm{mod}~\beta $, by what was said above we find a desired solution to \eqref{eq:alpha_squeeze} and are done. If $|(X-X)\,\mathrm{mod}~\beta|=2\lfloor\tfrac 43 \lambda\beta\rfloor - 1$, we consider the set of numbers $x\in \ZZ\,\mathrm{mod}~\beta$ that do \emph{not} satisfy \eqref{eq:alpha_squeeze}. Recall that by Lemma \ref{lem:squeeze}, the size of this set is $2\beta - 2 \lfloor \tfrac 43 (1-\lambda)\beta\rfloor - 1$. Towards a contradiction, assume that
\[
|(X-X)\,\mathrm{mod}~\beta| = 2\lfloor\tfrac 43 \lambda\beta\rfloor - 1 \leq 2\beta - 2 \lfloor \tfrac 43 (1-\lambda)\beta\rfloor - 1.
\]
This is equivalent to
\[
\lfloor \tfrac 43 \lambda \beta \rfloor + \lfloor \tfrac 43 (1-\lambda)\beta\rfloor \leq \beta.
\]
However, since $\beta = c \geq 6$, we have
\[
\lfloor \tfrac 43 \lambda \beta \rfloor + \lfloor \tfrac 43 (1-\lambda)\beta\rfloor > \tfrac 43 \lambda \beta - 1 + \tfrac 43 (1-\lambda)\beta - 1 = \tfrac 43 \beta -2 \geq \beta.
\]
Hence, there are more elements in $(X-X)\,\mathrm{mod}~\beta$ than there are counterexamples to \eqref{eq:alpha_squeeze} in $\ZZ\,\mathrm{mod}~\beta$. So we find a solution to \eqref{eq:alpha_squeeze} in $(X-X)\,\mathrm{mod}~\beta$, which concludes the case $c\geq 6$.

It remains to treat the cases $c\in \{1,\dots,5\}$. In these cases, there is not a sufficiently big paralellogram in the Minkowski convex combination, therefore we consider more integer points coming from the difference bodies of the original tiles. 
By Lemma \ref{lem:cayley_covering}, it suffices to show that the following Cayley sum is covering:
\[
Q:=\left[\frac{4}{3}\left( \frac{1}{2}Q_1+\frac{1}{2}Q_2\right), \frac{4}{3}Q_1\right].
\]
If $c=1$, all integer vectors in $Q_2-Q_2$ are in $H_0$. By Lemma \ref{lemma:parity}, $(Q_2-Q_2)\cap \ZZ^2= H_0\cap \ZZ^2$ and therefore the covering body $[\frac{2}{3}H_0, \frac{2}{3}H_0]$ is contained in $Q$.\\
If $c=2$, $(Q_2-Q_2)\cap \ZZ^2\not\subset H_0$ and $(Q_2-Q_2)\cap \ZZ^2\subset 2H_0$. By symmetry of $H_0$ and the fact that $Q_2-Q_2$ only contains primitive lattice vectors, we can assume $(1,1)\in (Q_2-Q_2)\cap \ZZ^2$. One can then check that the only possible $\GL_2(\ZZ)$ transformations of $H_0$ contained in $Q_2-Q_2$ are $H_0$ and $H_1:=\conv \{ \pm e_1, \pm e_2, \pm(e_1+e_2) \}$. It remains to check that $[\frac{2}{3}H_1, \frac{2}{3}H_0]$ is covering: indeed, then by Lemma \ref{lem:cayley_covering}, so are any Minkowski combinations of $\frac{2}{3}H_0\subset Q_1$ and $\frac{2}{3}H_1\subset Q_2$. This is done via a \texttt{sagemath} computation, discussed below.\\
Finally, let $c\in \{3,4, 5\}$. Similarly to the case $c\geq 6$, we can assume the existence of a vector $v=(\alpha, \beta)\in (Q_2-Q_2)\cap \ZZ^2$, with $0< -2\alpha <  \beta= \tfrac 12 c$ and $\gcd(\alpha, \beta)=1$, which further implies $\frac{1}{2}[-v,v]\subset Q_2$. Then, it is enough to show that
\[
\left[ \frac{1}{3}\left(H_0+[-v,v]\right), \frac{2}{3}H_0 \right]\subset Q
\]
is covering. The possibilities for $v$ are:
\[v\in \{ (-1,3), (-1,4), (-1,5), (-2,5)\}.\]
For these cases, a \texttt{sagemath} \cite{sagemath} computation shows that $\left[\frac{1}{3}\left(H_0+[-v,v]\right), \frac{2}{3}H_0 \right]$ is covering. The computation is an inclusion/exclusion based calculation of the volume of this lattice arrangement intersected with the unit cube. The computations can be accessed via the following URL: \url{https://github.com/AnsgarFreyer/covering_brunn-minkowski}
\end{proof}
\noindent Notice that in the previous proof, the case $c=1$ follows from a stronger claim, namely Corollary \ref{cor:planar_cayley_covering}.
Finally, we extend this inequality to all convex bodies in the plane.
\begin{theorem}
\label{thm:planar_mu_bm_text}
For convex bodies $K_1,K_2\subseteq \RR^2$, the following inequality holds:
\[
\mu(K_1+K_2)^{-1}\geq \tfrac{3}{4}\left( \mu(K_1)^{-1}+\mu(K_2)^{-1}\right).
\]
Moreover, equality is attained, e.g., for the parallelograms $P_1$ and $P_2$ from \eqref{eq:parallelograms}.
\end{theorem}

\begin{proof}
Since the inverse of the covering radius is $1$-homogeneous, it suffices to show 
\[\mu\left(\frac{\mu(K_1)^{-1}}{\mu(K_1)^{-1}+\mu(K_2)^{-1}} \mu(K_1)K_1+\frac{\mu(K_2)^{-1}}{\mu(K_1)^{-1}+\mu(K_2)^{-1}} \mu(K_2)K_2 \right)^{-1}\geq \frac{3}{4}.\]
Since the convex bodies $\mu(K_1)K_1$ and $\mu(K_2)K_2$ are covering bodies in $\RR^2$, they each contain convex tiles. Since the inverse of the covering radius is monotonously increasing with respect to inclusion, the inequality follows from Lemma \ref{lem:1.5}.
That the parallelograms $P_1$ and $P_2$ achieve equality can be seen from elementary geometric considerations.
\end{proof}

Note that the proof of Theorem \ref{thm:planar_mu_bm_text} is not dependent on the dimension in the following sense -- if one were to prove that for all minimal covering bodies a Brunn-Minkowski type inequality \eqref{eq:mu_bm_prototype} holds for some $c_d$, the result would extend to all convex bodies in $\RR^d$ using Proposition \ref{prop:zorn}.\\
On the other hand, showing 
\[
\mu((1-\lambda)Q_1+\lambda Q_2)\geq (1-\lambda)\mu(Q_1)+\lambda \mu(Q_2)=1\]
for certain minimal covering bodies $Q_1$ and $Q_2$  would yield that their Cayley sum is minimal covering by Propositions \ref{prop:tensor} and \ref{prop:wiggly_sums}. 
In this sense, the connection between Brunn-Minkowski type inequalities for the inverse covering radius and minimal covering bodies is twofold.
\subsection{A remark on arbitrary dimensions}

We provide an argument that refines the trivial estimate $c_d > \tfrac 12$ in \eqref{eq:mu_bm_prototype}. While it does not suffice to provide a better lower bound on $c_d$, it leads to some questions that might be of independent interest. 

\begin{lemma}
\label{lemma:inclusion}
Let $A,B\subset\RR^d$ be origin symmetric convex bodies such that $A\subseteq B$ and $\inter(A)\neq \emptyset$. Then,
\[
B\subseteq d\cdot \frac{\vol(B)}{\vol(A)} A.
\]
\end{lemma}

\begin{proof}
Let $\lambda \geq 0$ be the minimal dilation factor such that $B\subset \lambda A$. Then there exists a common boundary point $x$ of $\lambda A$ and $B$. By the pyramid formula and the origin symmetry of $B$ we have
\begin{equation}
\label{eq:pyramid}
\vol_d(B) \geq \frac{2 \Vert x\Vert}{d}\vol_{d-1}(B\cap x^\perp).
\end{equation}
On the other hand, by the Brunn-Minkowski inequality, the function $h\mapsto \vol((\lambda A) \cap \{\langle x,\cdot\rangle = h\})$ is log-concave on its support. Hence,by the origin symmetry of $A$,
\[
\lambda^d \vol_d(A) = \vol_d(\lambda A) \leq 2 \Vert x \Vert \vol_{d-1}(\lambda A \cap x^\perp)  \leq \lambda^{d-1} 2\Vert x\Vert \vol_{d-1}(A\cap x^\perp).
\]
Since $A\subseteq B$ it follows with \eqref{eq:pyramid} that
\[
\lambda \leq 2\Vert x\Vert \frac{\vol_{d-1}(B\cap x^\perp)}{\vol_d(A)} \leq d\cdot  \frac{\vol_d(B)}{\vol_d(A)},
\]
which concludes the proof.
\end{proof}

Consider two convex tiles $Q_1,Q_2\subset \RR^d$ centered in the origin. Then Lemma \ref{lemma:inclusion} yields that
\(
\tfrac{\vol(Q_1\cap Q_2)}{d}Q_2\subseteq Q_1.
\)
Consequently, we have
\begin{equation}
\label{eq:tile_inclusion}
\mu(Q_1 + Q_2)^{-1} \geq  1 + \frac{\vol(Q_1\cap Q_2)}{d}.
\end{equation}
In general, the volume of the intersection on the right hand side could be arbitrarily small. However, if we find ourselves in the setting of Theorem \ref{thm:tile_covering}, we can work with
the origin symmetric lattice polytope
\[
P_{Q_1} = \conv((Q_1-Q_1)\cap \ZZ^d) \subseteq (Q_1-Q_1)\cap (Q_2-Q_2).
\]
Let $\nu_d = \max\{\vol(P_{Q}) : Q\subset \RR^d~\text{convex tile}\}$, then, whenever $(Q_1 - Q_1)\cap\ZZ^d \subset (Q_2 - Q_2)\cap\ZZ^d$, we have
\[
\mu(Q_1 + Q_2)^{-1} \geq 1 + \tfrac {2^{-d}\nu_d}{d} \geq 1 + \tfrac 1 {d\cdot d!},
\]
where the last inequality follows from the observation that $\nu_d \geq 2^{d}/d!$.

It seems plausible that $\nu_d = d+1$, which in turn would yield a lower bound of $1 + (1+\tfrac 1d)2^{-d}$ on the covering radius of the Minkowski sum.

We conclude by showing that $\vol(P_Q) \geq d+1$ holds in the case where $Q$ is a Voronoi cell with respect to some positive definite quadratic form. We refer to \cite{vallentin} for an excellent exposition of Voronoi and Delaunay tessellations by quadratic forms.
%$q$, i.e.,
%\[
%Q = \{x\in\RR^d : q(x) \leq q(x-z),~\forall z\in\ZZ^d\}.
%\]
\begin{proposition}
Let $Q$ be a Voronoi cell with respect to a quadratic form. Then $\vol(P_Q) \geq d+1.$
\end{proposition}

\begin{proof}
Let $q$ be the quadratic form that defined $Q$ and let $S$ be the union of the Delaunay polytopes defined by $q$ that contain the origin. By the duality between Delaunay and Voronoi decompositions, we see that $S\subset Q-Q$.

If $q$ is generic, i.e., the Delaunay decomposition that it induces is a triangulation, we see that $S$ is a $(d+1)$-fold covering of $\RR^d$. Hence, $\vol(S)= d+1$. If $q$ is not generic, one finds a quadratic form $q'$ whose Delaunay triangulation refines the Delaunay decomposition of $q$. Let $S'$ be the corresponding union of Delaunay simplices for $q'$. We have $S' \subseteq S$ and, thus, $\vol(S) \geq d+1$.
\end{proof}

\bibliographystyle{abbrv}
\bibliography{literature}
\end{document}